\documentclass[final,leqno,onetabnum]{siamltex704}

\usepackage{amsmath,amssymb} 
\usepackage{amsfonts} 
\usepackage{exscale}
\usepackage{graphicx}
\usepackage{epstopdf}  
\usepackage{epsfig}

\usepackage[%
  breaklinks=true,%
  colorlinks=true,%
  linkcolor=blue,anchorcolor=blue,%
  citecolor=blue,filecolor=blue,%
  menucolor=blue,
  urlcolor=blue]{hyperref}
\usepackage{url,doi}

\renewcommand{\ldots}{\ensuremath{\dotsc}}
\newcommand{\comment}[1]{}

\newcommand{\R}{\mathbb{R}} 
\newcommand{\C}{\mathbb{C}} 

\def\Z{{\mathcal Z}}

\newcommand{\X}{\mathcal{X}}
\newcommand{\Y}{\mathcal{Y}}
\newcommand{\F}{\mathcal{F}}
\newcommand{\G}{\mathcal{G}}
\newcommand{\XP}{{\mathcal{X}^\perp}}
\newcommand{\YP}{{\mathcal{Y}^\perp}}
\newcommand{\QX}{X}
\newcommand{\QXP}{X_\perp}
\newcommand{\QY}{Y}
\newcommand{\QYP}{Y_\perp}

\newcommand{\PP}{\mathcal{P}}
\newcommand{\QQ}{\mathcal{Q}}
\newcommand{\lp}{\left(}
\newcommand{\rp}{\right)}

\def\span{{\rm span}}
\def\dim{{{\rm dim}}}
\def\max{{{\rm max}}}
\def\min{{{\rm min}}}

\newcommand{\re}{\mathop{\mathrm{re}}}
\newcommand{\pol}{\mathop{\mathrm{pol}}}

\newtheorem{thm}{Theorem}[section]
\newtheorem{cor}[thm]{Corollary}

\newtheorem{conjecture}[thm]{Conjecture}
\newtheorem{defin}[thm]{Definition}

\newtheorem{remark}[thm]{Remark}

\title{Rayleigh-Ritz majorization error bounds\\ of the mixed type
\thanks{%
A preliminary version is posted at arXiv \cite{kz16}. 
}}

\author{%
Peizhen Zhu\footnotemark[2]\ \footnotemark[3]
\and
Andrew V. Knyazev\footnotemark[4]\ \footnotemark[5]\ \footnotemark[6]
}

\begin{document}
\vspace{-1.2in}
\vspace{.9in}

\setcounter{page}{1}
\maketitle

\renewcommand{\thefootnote}{\fnsymbol{footnote}}
\footnotetext[2]{Department of Mathematical \& Statistical; 
Missouri University of Science and Technology,
202 Rolla Building, 400 W. 12th St., Rolla, MO 65409-0020
}
\footnotetext[3]{zhupe[at]mst.edu}
\footnotetext[4]{Mitsubishi Electric Research Laboratories; 201 Broadway
Cambridge, MA 02139}
\footnotetext[5]{knyazev[at]merl.com}
\footnotetext[6]{\url{http://www.merl.com/people/knyazev}
}
\renewcommand{\thefootnote}{\arabic{footnote}}

\begin{abstract}
The absolute change in the Rayleigh quotient (RQ)  for a Hermitian matrix with respect to vectors is bounded in
terms of the norms of the residual vectors and the angle between vectors in [\doi{10.1137/120884468}]. We substitute multidimensional subspaces for the vectors and derive new bounds of absolute changes of eigenvalues of the matrix RQ in terms of singular values of residual matrices and principal angles between subspaces, using majorization.
We show how our results relate to bounds for eigenvalues after discarding off-diagonal blocks
or additive perturbations.

\end{abstract}
\begin{keywords}
Rayleigh, Ritz, majorization, angles, subspaces, eigenvalue, Hermitian, matrix
\end{keywords}
\begin{AM}
  15A03, 15A18, 15A42, 15B57.
\end{AM}


\pagestyle{myheadings}
\thispagestyle{plain}
\markboth{PEIZHEN ZHU and ANDREW V. KNYAZEV}
{Rayleigh-Ritz majorization mixed error bounds} 

\section{Introduction}\label{intro}
In this work, we continue our decades-long 
investigation, see \cite{akpp08,k85,k86,k87,ka06a,ka06,ka10,kja10,ko06,zhu2013}, of sensitivity of the Rayleigh quotient (RQ)  and the Ritz values (the eigenvalues of the matrix RQ)  with respect to changes in vectors or subspaces, and closely related error bounds for accuracy of eigenvalue approximations by the Rayleigh-Ritz (RR)  method, for  Hermitian matrices and operators. 

There are two main types of bounds of absolute changes of the Ritz values: \emph{a priori} and \emph{a posteriori} bounds, as classified, e.g.,\ in  \cite{zhu2013}.
The~a~priori bounds, e.g.,\ presented in \cite{ka10},
are in terms of principal angles between subspaces (PABS) , which may not be always readily available in practice. 
The~a~posteriori bounds, e.g.,\ presented in \cite{bhatia_book,Mathias98,parlett80,sunste}, are based on easily computable singular values of residual matrices.  
Our  bounds in this work use both PABS and the singular values of the residual matrices, thus called \emph{mixed bounds}, following \cite{zhu2013}. 

Different vectors/subspaces may have the same RQ/Ritz values, but if the 
values are different, the vectors/subspaces cannot be the same, evidently.  
A priori and mixed bounds can be used to differentiate vectors/subspaces, providing guaranteed lower bounds
for the angles. In \cite{ka06}, this idea is illustrated for graph matching, by comparing spectra of graph Laplacians of the graphs to be matched. Our bounds can be similarly applied for signal distinction in signal processing, where the Ritz values serve as a harmonic signature of the subspace; cf. \emph{Star Trek subspace signature}.
Furthermore, the RQ/Ritz values are computed independently for every vector/subspace in a pair, thus, also suitable for distributed or privacy-preserving data mining, while determining the angles requires both vectors/subspaces in a pair to be available.

The rest of this paper is organized as follows. Section \ref{sec:mc} motivates 
our work and contains our conjecture.
In Section \ref{sec:majorization}, we formally introduce the notation, define majorization, and formulate PABS.
Section \ref{sec:ch6.3} contains our main results---mixed bounds of 
the absolute changes of the Ritz values in terms of PABS and the singular values of the residual matrices by using weak majorization. In Section \ref{sec:discussion},  we compare our mixed majorization bounds with those known and relate to eigenvalue perturbations.

\section{Motivation and Conjectures}\label{sec:mc}
For a nonzero vector $x$ and a Hermitian matrix $A$, RQ, defined by
$\rho\lp x\rp =\left\langle x, Ax\right\rangle/\left\langle x,x\right\rangle$, where $\left\langle\cdot,\cdot \right\rangle$ is an inner product, associated with a norm by $\|\cdot\|^2=\left\langle\cdot,\cdot\right\rangle$, is typically used as an approximation to some eigenvalue of $A$. The corresponding residual vector is denoted by $r\lp x\rp =Ax-\rho\lp x\rp x$.  

Let $x$ and $y$ be unit vectors. The angle $\theta\lp x,y\rp $ between them is defined by $\cos\theta\lp x,y\rp =\left| \langle x,y \rangle \right|$. 
In  \cite[Theorem 3.7 and Remark 3.6]{zhu2013}, the absolute change in RQ with respect to the vectors $x$ and $y$, not orthogonal to each other,  is bounded as follows,
\begin{eqnarray}\label{eqn:chap6_1}
|\rho\lp x\rp -\rho\lp y\rp |&\leq&\frac{\|P_\Y\, r\lp x\rp \|+\|P_\X\, r\lp y\rp \|}{\cos\theta\lp x, y\rp }\\
&=& \lp  \|P_{\X+\Y}\, r\lp x\rp \|+\|P_ {\X+\Y}\, r\lp y\rp \|\rp \tan\theta\lp x, y\rp , \nonumber
\end{eqnarray}
where  $P_\X$ and $P_\Y$ are orthogonal projectors on the one-dimensional subspaces $\X$ and $\Y$ spanned by the vectors $x$ and $y$, correspondingly, and 
$P_{{\X+\Y}}$ is the orthogonal projector on the subspace spanned by vectors $x$ and $y$.
If the vector $x$ is an eigenvector of $A$, the RQ of $x$ is an eigenvalue of~$A$, and \eqref{eqn:chap6_1} turns into the following equalities, 
\begin{equation}\label{eqn:chap6_2}
|\rho\lp x\rp  -\rho\lp y\rp |=\frac{\|P_\X\, r\lp y\rp \|}{\cos\theta\lp x, y\rp }=\|P_ {\X+\Y}\, r\lp y\rp \|\tan\theta\lp x, y\rp ,
\end{equation}
since $r\lp x\rp =0$; i.e. the absolute change  $|\rho\lp x\rp -\rho\lp y\rp |$ of the RQ
becomes in \eqref{eqn:chap6_2} the absolute error in the eigenvalue $\rho\lp x\rp $ of $A$; cf. \cite{knyazev97,sun91,zhu2013}. 

It is elucidative to examine bounds \eqref{eqn:chap6_1} and \eqref{eqn:chap6_2} in a context of an asymptotic Taylor expansion of the RQ at the vector $y$ with the expansion center at the vector~$x$. If $x$ is  an eigenvector of $A$, then $r\lp x\rp =0$, the gradient of the RQ at $x$, and thus the first order term in the Taylor expansion vanishes. This implies that the RQ behaves as a quadratic function in a vicinity of the eigenvector $x$, e.g., giving a second order bound for the  absolute change  $|\rho\lp x\rp -\rho\lp y\rp |$ of the RQ, as captured by \eqref{eqn:chap6_2} as well as by the following a priori bound from \cite[Theorem 4]{ka06a},
\begin{equation}\label{eqn:2order}
|\rho\lp x\rp -\rho\lp y\rp |=\left[\lambda_{\max}-\lambda_{\min}\right]\sin^2\theta\lp x, y\rp ,
\end{equation}
where $\lambda_{\max}\geq \lambda_{\min}$ are the two eigenvalues of the projected matrix $P_{\span\lp x\rp +\span\lp y\rp }A$ restricted to 
its invariant two-dimensional subspace $\span\lp x\rp +\span\lp y\rp $, which is the range of the orthogonal projector $P_{\span\lp x\rp +\span\lp y\rp }$.

If $x$ is not an eigenvector of $A$, then the gradient of the RQ at $x$ is not zero, and one can obtain only a first order bound, e.g., as in \cite[Theorem 1 and Remark 3]{ka06a},
\begin{equation}\label{eqn:1order}
|\rho\lp x\rp -\rho\lp y\rp |=\left[\lambda_{\max}-\lambda_{\min}\right]\sin\theta\lp x, y\rp .
\end{equation}
 
If the vector $x$ moves toward an eigenvector of $A$, the first order term in the Taylor expansion gets smaller, so a desired first order bound can be expected to gradually turn into a second order one, which bound \eqref{eqn:chap6_1} demonstrates turning into \eqref{eqn:chap6_2}, in contrast to autonomous a priori bounds \eqref{eqn:2order} and \eqref{eqn:1order}.

In this paper, we substitute finite dimensional subspaces $\X$ and $\Y$ of the same dimension
for one-dimensional subspaces spanned by the vectors $x$ and $y$ and, thus, generalize bounds \eqref{eqn:chap6_1} and \eqref{eqn:chap6_2} to the multi-dimensional case.
An extension of RQ to the multi-dimensional case is provided by the Rayleigh-Ritz (RR)  method; see, e.g.,\ \cite{{parlett80},{sunste}}.
Specifically, let columns of the matrix $\QX$ form an orthonormal basis for the subspace~$\X$. 
The matrix RQ associated with the matrix $\QX$ is defined by
$\rho\lp X\rp =X^HAX,$
and the corresponding residual matrix is $R_X=AX-X\rho\lp X\rp.$ 

Since the matrix $A$ is Hermitian, the matrix RQ $\rho\lp X\rp $ is also Hermitian. 
The eigenvalues of $\rho\lp X\rp $  do not depend on the particular choice of the basis $X$ of the subspace $\X$,
and are commonly called ``Ritz values'' of the matrix $A$ corresponding to the subspace $\X$.
If~the subspace $\X$ is $A$-invariant, i.e. $A\X \subset \X$, then $R_X=0$ and all eigenvalues of $\rho\lp X\rp $ are also the 
eigenvalues of $A$.

Our goal is bounding changes in Ritz values where the subspace varies. Particularly, we bound the 
differences of eigenvalues of Hermitian matrices $\rho\lp X\rp =X^HAX$ and $\rho\lp Y\rp =Y^HAY$, where the columns of the  matrices $X$ and  $Y$ form orthonormal bases for  subspaces $\X$ and $\Y$, correspondingly.
 In particular, if the subspace $\X$ is an invariant subspace of $A$, 
then the changes of eigenvalues of Hermitian matrices $X^HAX$ and $Y^HAY$ represent  approximation errors of eigenvalues of the Hermitian matrix $A$.
We generalize \eqref{eqn:chap6_1} and \eqref{eqn:chap6_2} to the multidimensional setting using 
majorization, see \cite{bhatia_book,{marshall79}}, as stated in the following conjecture.
\begin{conjecture}\label{thm:conjecture}
Let columns of matrices $X$ and $Y$ form orthonormal bases for the subspaces $\X$ and $\Y$ 
with $\dim\lp \X\rp =\dim\lp \Y\rp $, correspondingly.
  Let $A$ be a Hermitian matrix and $\Theta\lp \X,\Y\rp <\pi/2$. Then 
\begin{eqnarray}\label{eqn:conjecture}
\left|\Lambda\lp X^HAX\rp -\Lambda\lp Y^HAY\rp \right|&\prec_{w}& 
 \frac{S\lp P_\Y R_X\rp +S\lp P_\X R_Y\rp }{\cos\lp \Theta\lp \X,\Y\rp \rp },
 \end{eqnarray}
\begin{eqnarray}\label{eqn:conjecture_tan}
\hspace{9mm} 
\left|\Lambda\lp X^HAX\rp -\Lambda\lp Y^HAY\rp \right|&\prec_{w}& 
 \left\{ S\lp P_{\X+\Y} R_X\rp +S\lp P_{\X+\Y}R_Y\rp \right\}\, \tan\lp \Theta\lp \X,\Y\rp \rp .
\end{eqnarray}
If the subspace $\X$ is $A$-invariant, then
\begin{eqnarray}\label{eqn:eigen_conjecture}
\left|\Lambda\lp X^HAX\rp -\Lambda\lp Y^HAY\rp \right|&\prec_{w}&\frac{S\lp P_\X R_Y\rp }{\cos\lp \Theta\lp \X,\Y\rp \rp },
\end{eqnarray}
\begin{eqnarray}\label{eqn:eigen_conjecture_tan}
\left|\Lambda\lp X^HAX\rp -\Lambda\lp Y^HAY\rp \right|&\prec_{w}&
S\lp P_{\X+\Y} R_Y\rp \, \tan\lp \Theta\lp \X,\Y\rp \rp .
\end{eqnarray}
\end{conjecture}
$P_\X$ and $P_\Y$ are orthogonal projectors on subspaces $\X$ and $\Y$; 
$\Lambda\lp \cdot\rp $ denotes the vector of decreasing eigenvalues; 
$S\lp \cdot\rp $ denotes the vector of decreasing singular values; 
$\Theta\lp \X,\Y\rp $ denotes the vector of decreasing angles between the subspaces $\X$ and $\Y$;
$\prec_{w}$~denotes the weak majorization relation; 
see the formal definitions in Section \ref{sec:majorization}. 
All arithmetic operations with vectors in Conjecture \ref{thm:conjecture}
are performed element-wise. 
Let us note that the eigenvalues and the singular values appearing in Conjecture \ref{thm:conjecture}
do not depend on the particular choice of the bases $X$ and $Y$ of the subspaces $\X$ and $\Y$, correspondingly. 

To highlight advantages of the mixed bounds of Conjecture \ref{thm:conjecture},  
 compared to a~priori majorization bounds, 
 we formulate here one known result as follows, cf. \cite{LI1992249}.
 \begin{thm}[{\cite[Theorem 2.1]{ka10}}]\label{thm:cha6priori}
Under the assumptions of Conjecture \ref{thm:conjecture}, 
\begin{eqnarray}\label{eqn:ap}
\left|\Lambda\lp X^HAX\rp -\Lambda\lp Y^HAY\rp \right|&\prec_w& \left[\lambda_{\max}-\lambda_{\min}\right]\sin\lp  \Theta\lp \X,\Y\rp \rp .
\end{eqnarray}
If in addition one of the subspaces is $A$-invariant, then
\begin{eqnarray}\label{eqn:eigen_ap}
\left|\Lambda\lp X^HAX\rp -\Lambda\lp Y^HAY\rp \right|&\prec_w& \left[\lambda_{\max}-\lambda_{\min}\right]\sin^2\lp  \Theta\lp \X,\Y\rp \rp ,
\end{eqnarray}
where $\lambda_{\max}\geq \lambda_{\min}$ are the end points of the spectrum of the matrix $P_{\X+\Y}A$ restricted to 
its invariant subspace $\X+\Y$.
 \end{thm}

Theorem \ref{thm:cha6priori} generalizes bounds \eqref{eqn:2order} and \eqref{eqn:1order} to the multidimensional case, but also inherits their deficiencies. 
Similar to bound \eqref{eqn:chap6_1} being compared to bounds \eqref{eqn:2order} and \eqref{eqn:1order}, Conjecture \ref{thm:conjecture} is more mathematically elegant, compared to Theorem \ref{thm:cha6priori}.
Indeed, bound \eqref{eqn:ap} cannot imply 
bound \eqref{eqn:eigen_ap} in Theorem \ref{thm:cha6priori}, while in Conjecture~\ref{thm:conjecture} bounds \eqref{eqn:eigen_conjecture} and \eqref{eqn:eigen_conjecture_tan} for the case of the $A$-invariant subspace $\X$ follow directly from bounds \eqref{eqn:conjecture} and \eqref{eqn:conjecture_tan}, since $R_X=0$ and some terms vanish. 

Conjecture \ref{thm:conjecture} is particularly advantageous in a case 
where both $\X$ and $\Y$ approximate the same $A$-invariant subspace, 
so that the principal angles between subspaces $\X$ and $\Y$ are small and 
the singular values of both residual matrices are also small, e.g., leading to bound \eqref{eqn:conjecture_tan}
of nearly the second order, while bound \eqref{eqn:ap} is of the first order.
For example, let the subspace $\X$ be obtained off-line to approximate an $A$-invariant subspace, 
while the subspace $\Y$ be computed from the subspace $\X$ by rounding in a low-precision computer arithmetic of components 
of the basis vectors spanning $\X$, for the purpose of efficiently storing, fast transmitting, or quickly analyzing in real time. 
Bounds \eqref{eqn:conjecture} and \eqref{eqn:conjecture_tan} allow estimating the effect of the rounding on the change in the Ritz values and choosing an optimal rounding precision. 

Another example is deriving convergence rate bounds of first order iterative minimization methods related to the Ritz values; e.g.,\ subspace iterations like the Locally Optimal Block Preconditioned Conjugate Gradient (LOBPCG)  method \cite{k2001}. 
The first order methods typically converge linearly. If the iterative subspace, which approximates an invariant subspace, is slightly perturbed for whatever purpose, that effects in even smaller changes in the Ritz values, e.g., according to \eqref{eqn:conjecture_tan}, preserving essentially the same rate of convergence. Such arguments appear in trace minimization  \cite{zli15}, where \cite[Lemma 5.1]{zli15} presents an inequality for the trace of the difference $\Lambda\lp X^HAX\rp -\Lambda\lp Y^HAY\rp $ that gradually changes, as \eqref{eqn:conjecture_tan}, from the first to the second order error bound.
Much earlier examples can be found, e.g., in \cite[Theorem 4.2]{k87}, where the Ritz vectors are substituted by their surrogate approximations, which are easier to deal with, slightly perturbing the subspace. 
Similarly, in \cite[Lemma 4.1]{bpk96}, the actual nonlinear subspace iterations are approximated by a linear scheme. Availability of bounds like \eqref{eqn:conjecture_tan} is expected to greatly simplify proofs and lead to better convergence theory of subspace iterations.

We are unable to prove  Conjecture \ref{thm:conjecture}, although it holds in all our numerical tests. 
Instead, we prove here slightly weaker results, which still 
generalize and improve bounds obtained in \cite{{knyazev97},{sun91}}, even for the case where one
subspace is $A$-invariant. Our bounds, originating from \cite[Section 6.3]{zhu2012thesis},
exhibit the desired behavior balancing between first and second order error terms.

We finally motivate our results of \S\ref{ss:ap}. The Ritz values are fundamentally related to PABS, as discovered in \cite{ka06}. 
For example, simply taking the matrix $A$ to be an orthogonal projector $Z$ to a subspace $\Z$ of the same dimension as 
$\X$ and $\Y$ turns the Ritz values into the cosines squared of PABS, e.g., 
\[
 \Lambda\lp X^HAX\rp  = \Lambda\lp X^HZX\rp  = \cos^2\lp  \Theta\lp \X,\Z\rp \rp .
\]
Thus, in this example,  Conjecture \ref{thm:conjecture} bounds changes in the PABS $\Theta\lp \X,\Z\rp $ compared to $\Theta\lp \Y,\Z\rp $,
extending and improving some known results, e.g., from our earlier works \cite{ka02,ka06a,ka06}, even in the particular case
$\X=\Z$, where $\Z$ is $A=Z$-invariant. 

It~is interesting to investigate a geometric meaning, in terms of PABS, of the singular values that appear in the bounds, such as
$S\lp P_{\X+\Y} R_Y\rp $. We leave such an investigation to future research, 
except for one simplified case, where the projector $P_{\X+\Y}$ is dropped,
in \S\ref{ss:ap} that discusses new majorization bounds for changes in matrix eigenvalues under additive perturbations.

 \section{Notation and Definitions}\label{sec:majorization}
 Throughout this paper, $S\lp A\rp $ denotes  a vector of decreasing singular values of a matrix $A$, such that $S\lp A\rp =[s_1\lp A\rp ,\ldots,s_n\lp A\rp ]$. $S^2\lp A\rp $ denotes the entry-wise square, i.e. $S^2\lp A\rp =[s_1^2\lp A\rp ,\ldots,s_n^2\lp A\rp ]$.  $\Lambda\lp A\rp $ denotes a vector of decreasing eigenvalues of $A$. $|||\cdot|||$ denotes a unitarily invariant norm. $s_{\min}\lp A\rp $ and $s_{\max}\lp A\rp =\|A\|$ denote the smallest and largest, correspondingly, singular values of $A$. $\lambda_{\min}\lp A\rp $ and $\lambda_{\max}\lp A\rp $ denote the smallest and largest eigenvalues of a Hermitian matrix $A$. $P_\X$~denotes an orthogonal projector on the subspace $\X$. $\XP$ denotes the orthogonal complement of the subspace $\X$. $\Theta\lp \X,\Y\rp $ denotes the vector of decreasing principal angles between the subspaces $\X$ and $\Y$.  $\theta_{\max}\lp \X,\Y\rp $ and $\theta_{\min}\lp \X,\Y\rp $ denote the largest and smallest angles between subspaces $\X$ and $\Y,$ correspondingly.  
 
 We use the symbols ``$\downarrow$'' and ``$\uparrow$'' to arrange components of a vector in decreasing and increasing order, correspondingly. For example, the equality $x^\downarrow=[x_1,\ldots,x_n]$ implies $x_1\geq x_2\geq \cdots\geq x_n$. All arithmetic operations with vectors are performed entry-wise, without introducing a special notation. 
 Vectors $S\lp \cdot\rp $, $\Lambda\lp \cdot\rp ,$  and $\Theta\lp \cdot\rp $ are by definition decreasing, i.e. $S\lp \cdot\rp =S^{\downarrow}\lp \cdot\rp $, $\Lambda\lp \cdot\rp =\Lambda^{\downarrow}\lp \cdot\rp ,$  and $\Theta\lp \cdot\rp =\Theta^{\downarrow}\lp \cdot\rp .$
 
 We now define the concepts
of weak majorization and (strong)  majorization, which are comparison relations between
two real vectors. For detailed information, we refer the reader to \cite{bhatia_book,marshall79}.
\begin{defin}
Let $x^{\downarrow}$ and $y^{\downarrow} \in R^n$ be the vectors obtained by rearranging the coordinates of vectors 
$x$ and $y$ in algebraically decreasing order, denoted by $x_1, \ldots , x_n$ and $y_1, \ldots, y_n$,
such that $x_1\geq x_2\cdots\geq x_n$ and $y_1\geq y_2\cdots\geq y_n.$ 
 We say that $x$ is weakly majorized by $y$, 
  using the notation $ x \prec_w y,$ if
\[
\sum^k_{i=1}x_i\leq \sum^k_{i=1}y_i,   \quad 1\leq k\leq n.
\]
If in addition $
\sum^n_{i=1}x_i = \sum^n_{i=1}y_{i},
$
we say that 
$x$ is  majorized or strongly majorized by $y$, using the notation $ x \prec y$.
\end{defin}

The inequality $x \leq y$ means $x_i\leq y_i$ for $i=1,\ldots,n$. Therefore, $x \leq y$ 
implies $x \prec_w y$, but $x \prec_w y$ does not imply $x \leq y$. The relations 
$\prec$ and $\prec_w$ are both reflective and transitive~\cite{{bhatia_book}}. 

\begin{thm}[{\cite{Polya50}}]\label{thm:convex} If $f$ is a increasing convex function, then
$x\prec_w y$ implies $f\lp x\rp \prec_w f\lp y\rp $.
\end{thm} 
 
 From Theorem \ref{thm:convex}, we see that increasing convex functions preserve the weak majorization.
The following two results also provide ways to preserve majorization.
Let nonnegative vectors $x, y, u,$ and $v$ be decreasing and of the same size. 
If $x \prec_w y$, then $x \,  u \prec_w y \,  u$; see, e.g.,\ \cite{ka10}. 
Moreover, if  $x \prec_w y$ and  $u \prec_w v$, then $x\,  u \prec_w y \,  v$. The proof is simple,
$x\prec_w y$ implies $x \,  u \prec_w y \,  u$ and  $u \prec_w v$ implies $y \,  u \prec_w y \,  v$, 
so we have $x\,  u \prec_w y \,  v$.

  Majorization is one of the most powerful techniques that can be used to 
derive inequalities in a concise way.    Majorization relations 
among eigenvalues and singular values
of matrices produce a variety of inequalities in matrix theory. 
 We review some existing majorization inequalities
and prove necessary new majorization inequalities for singular
values and eigenvalues in the Appendix.

We define PABS using singular values; see, e.g.,\ \cite{bjorck73,ka02,ZhuK13}. 
\begin{defin}\label{thm:2.1def}
Let columns of the matrices $\QX \in \C^{n \times p}$ and 
$\QY \in \C^{n \times q}$ 
form  orthonormal bases for the subspaces $\X$ and $\Y$, 
correspondingly, and $m=\min\lp p,q\rp $. 
Then
\[
\cos\lp  \Theta^{\uparrow}\lp \X,\Y\rp \rp =S\lp  \QX^H\QY\rp =
\left[s_1\lp  \QX^H\QY\rp , \ldots, s_m\lp  \QX^H\QY\rp \right].
\] 
\end{defin}

\section{Majorization-type mixed bounds}\label{sec:ch6.3}
In this section, we derive several different majorization-type mixed bounds of the absolute changes of  eigenvalues of the matrix RQ for Hermitian matrices in terms of 
singular values of residual matrix and PABS.
One of our main results is contained in the following theorem.
\begin{thm}\label{thm:mixsincos}
Under the assumptions of Conjecture \ref{thm:conjecture},
we have  
\begin{eqnarray}\label{eqn:ch6_1}
\hspace{9mm}
\left|\Lambda\lp X^HAX\rp -\Lambda\lp Y^HAY\rp \right|  &\prec_w
&\frac{1}{\cos\lp \theta_{\max}\lp \X,\Y\rp \rp }\left\{S\lp P_\Y R_X\rp +S\lp P_\X R_Y\rp \right\}.
\end{eqnarray}
If the subspace $\X$ is $A$-invariant, then
\begin{eqnarray}\label{eqn:ch6_2}
\left|\Lambda\lp X^HAX\rp -\Lambda\lp Y^HAY\rp \right|  &\prec_w
&\frac{1}{\cos\lp \theta_{\max}\lp \X,\Y\rp \rp }S\lp P_\X R_Y\rp .
\end{eqnarray}
\end{thm}
\begin{proof}
Since $\Theta\lp \X,\Y\rp <\pi/2$, the singular values of  $X^HY$ are positive, 
which implies that the matrix $X^HY$ is invertible.
We apply the first statement of Lemma~\ref{lemma:commutor2} with $A: =X^HAX$, $B: =Y^HAY$ and $T: =X^HY$ 
obtaining
\begin{eqnarray}\label{eqn:maj7}
\left|\Lambda\lp X^HAX\rp -\Lambda\lp Y^HAY\rp \right|\prec_w \frac{1}{s_{\min}\lp X^HY\rp }S\lp  X^HAX\lp X^HY\rp -\lp X^HY\rp Y^HAY\rp .\nonumber \\
\end{eqnarray}
By Definition \ref{thm:2.1def}, the singular values of 
$X^HY$ are the cosines of principal angles between two subspaces $\X$ and $\Y$. So, we have
\begin{eqnarray}\label{eqn:maj4}
 s_{\min}\lp X^HY\rp &=&\cos\lp \theta_{\max}\lp \X,\Y\rp \rp .
 \end{eqnarray}
Additionally, the expression $X^HAX\lp X^HY\rp -\lp X^HY\rp Y^HAY$ in the right side of \eqref{eqn:maj7} can be written as 
\begin{eqnarray}\label{eqn:maj6}
X^HAX\lp X^HY\rp -\lp X^HY\rp Y^HAY&=&X^HA\lp I-\QXP\QXP^H\rp Y-X^H\lp I-\QYP\QYP^H\rp AY \nonumber\\
&=&-X^HA\QXP\QXP^HY+X^H\QYP\QYP^HAY,
\end{eqnarray}
where $[\QX, \QXP] $ and $[\QY, \QYP]$ are unitary matrices.
Since the singular values are invariant under conjugate transpose and orthonormal transforms, we have
\begin{eqnarray}\label{eqn:prothm}
S\lp X^HA\QXP\QXP^HY\rp &=&S\lp Y^H\QXP\QXP^HAX\rp =S\lp P_\Y P_\XP AX\rp 
=S\lp P_\Y R_X\rp .
\end{eqnarray}
Similarly, $S\lp X^H\QYP\QYP^HAY\rp =S\lp P_\X R_Y\rp .$ 
From Theorem \ref{thm:fan} and taking into account equalities \eqref{eqn:maj6} and \eqref{eqn:prothm}, we establish that
\begin{eqnarray}\label{eqn:maj8}
S\lp X^HAX\lp X^HY\rp -\lp X^HY\rp Y^HAY\rp &=&S\lp -X^HA\QXP\QXP^HY+X^H\QYP\QYP^HAY\rp \nonumber\\
&\prec_w &S\lp X^HA\QXP\QXP^HY\rp +S\lp X^H\QYP\QYP^HAY\rp \nonumber\\
&=&S\lp P_\Y R_X\rp +S\lp P_\X R_Y\rp .
\end{eqnarray}
Substituting \eqref{eqn:maj8} and \eqref{eqn:maj4} into \eqref{eqn:maj7}, 
we obtain \eqref{eqn:ch6_1}.
If the subspace $\X$ is $A$-invariant, then $R_X=0$, and \eqref{eqn:ch6_1} turns into \eqref{eqn:ch6_2}.
\end{proof}


 
Let us clarify implications of the weak majorization inequalities in Theorem \ref{thm:mixsincos}. 
The components of both vectors $\Lambda\lp X^HAX\rp $ and $\Lambda\lp Y^HAY\rp $ are ordered decreasing. Let us denote 
$\Lambda\lp X^HAX\rp =\left[\alpha_1, \alpha_2, \ldots, \alpha_p\right]$ where $\alpha_1\geq\alpha_2\geq \ldots\geq\alpha_p$ and
 $\Lambda\lp Y^HAY\rp =\left[\beta_1, \beta_2, \ldots, \beta_p\right]$ where $\beta_1\geq \beta_2\geq \ldots\geq \beta_p$.
For $k=1, \ldots, p$, the weak majorization inequalities in \eqref{eqn:ch6_1} and \eqref{eqn:ch6_2} in Theorem \ref{thm:mixsincos}  are equivalent to
\begin{eqnarray}\label{eqn:implication_1}
\sum_{i=1}^k|\alpha_i-\beta_i|^{\downarrow}\leq \frac{1}{\cos\lp \theta_{\max}\lp \X,\Y\rp \rp }\sum_{i=1}^k\lp  s^{\downarrow}_i\lp  P_\Y R_X\rp +s^{\downarrow}_i\lp  P_\X R_Y\rp \rp ,
\end{eqnarray}
and
\begin{eqnarray}\label{eqn:implication_2}
\sum_{i=1}^k|\alpha_i-\beta_i|^{\downarrow}\leq \frac{1}{\cos\lp \theta_{\max}\lp \X,\Y\rp \rp }\sum_{i=1}^k\lp  s^{\downarrow}_i\lp  P_\X R_Y\rp \rp .
\end{eqnarray}
If $p=1$, the results in \eqref{eqn:implication_1} and \eqref{eqn:implication_2} are the same as  \eqref{eqn:chap6_1} and \eqref{eqn:chap6_2}. 

We use only the largest angle in Theorem \ref{thm:mixsincos}. We now prove two majorization mixed 
bounds involving all principal angles, but not as strong as \eqref{eqn:conjecture}. The first one is for the squares. 

\begin{thm}\label{thm:squaretan_1}
 Under the assumptions of Conjecture \ref{thm:conjecture},
we have 
\[
\left|\Lambda\lp X^HAX\rp -\Lambda\lp Y^HAY\rp \right|^2  \prec_w
\frac{\left\{ S\lp P_\Y R_X\rp +S\lp P_\X R_Y\rp \right\}^2}{\cos^2\lp \Theta^\downarrow\lp \X,\Y\rp \rp }.
\]
In addition, if the subspace $\X$ is $A$-invariant, then
\[
\left|\Lambda\lp X^HAX\rp -\Lambda\lp Y^HAY\rp \right|^2 \prec_w
\frac{S^2\lp P_\X R_Y\rp }{\cos^2\lp \Theta^\downarrow\lp \X,\Y\rp \rp }.
\]
\end{thm}
\begin{proof}
We substitute $X^HAX$, $Y^HAY$ and $X^HY$ for $A, B$, and $T$ in the 
third result of Lemma \ref{lemma:commutor2}, and get
\begin{eqnarray*}
&&\left|\Lambda\lp X^HAX\rp -\Lambda\lp Y^HAY\rp \right|^2 \\
&&\prec_w S^2\lp  \lp  X^HY\rp ^{-1}\rp \,  S^2\lp  X^HAX\lp X^HY\rp -\lp X^HY\rp Y^HAY\rp .
\end{eqnarray*}
Definition \ref{thm:2.1def} gives $ S\lp X^HY\rp =\cos\lp \Theta^{\uparrow}\lp \X,\Y\rp \rp $, thus,
\begin{eqnarray}\label{eqn:tansq1}
S\lp  \lp  X^HY\rp ^{-1}\rp =\frac{1}{\cos\lp \Theta^{\downarrow}\lp \X,\Y\rp \rp }.
\end{eqnarray}
 From \eqref{eqn:maj8}, we already have 
\[
S\lp  X^HAX\lp X^HY\rp -\lp X^HY\rp Y^HAY\rp \prec_w \left\{ S\lp P_\Y R_X\rp +S\lp P_\X R_Y\rp \right\}.
\]
Since increasing convex functions preserve weak majorization by Theorem \ref{thm:convex}, we take the function 
$f\lp x\rp =x^2$ for nonnegative $x$. Squaring both sides of the weak majorization inequality above yields 
\[
S^2\lp  X^HAX\lp X^HY\rp -\lp X^HY\rp Y^HAY\rp \prec_w \left\{ S\lp P_\Y R_X\rp +S\lp P_\X R_Y\rp \right\}^2.
\]
Together with \eqref{eqn:tansq1}, this proves the first statement of Theorem \ref{thm:squaretan_1}.
If the subspace $\X$ is $A$-invariant, then $R_X=0$, which completes the proof.
\end{proof}

Let us highlight, that one cannot take the square root from both sides of the weak majorization inequalities in Theorem \ref{thm:squaretan_1}. 
Without the squares, we can prove bound \eqref{eqn:conjecture} of Conjecture \ref{thm:conjecture}
but with an extra multiplier.

\begin{thm}\label{thm:conditiontan2} 
 Under the assumptions of Conjecture \ref{thm:conjecture}, we have 
\[
\left|\Lambda\lp X^HAX\rp -\Lambda\lp Y^HAY\rp \right|  \prec_w
\sqrt{c}\, \frac{\left\{ S\lp P_\Y R_X\rp +S\lp P_\X R_Y\rp \right\}}{\cos\lp \Theta^\downarrow\lp \X,\Y\rp \rp },
\]
where $c=\cos\lp \theta_{\min}\lp \X,\Y\rp \rp /\cos\lp \theta_{\max}\lp \X,\Y\rp \rp $. If the subspace $\X$ is $A$-invariant, then
\[
\left|\Lambda\lp X^HAX\rp -\Lambda\lp Y^HAY\rp \right|\prec_{w}\sqrt{c}\,\frac{S\lp P_\X R_Y\rp }{\cos\lp \Theta^\downarrow\lp \X,\Y\rp \rp }.
\]
\end{thm}
\begin{proof}
The assumption $\Theta\lp \X,\Y\rp <\pi/2$ implies that $X^HY$ is invertible, so $Y^HAY$ is similar to the matrix $\lp  X^HY\rp Y^HAY\lp  X^HY\rp ^{-1}$.
The matrix $A$ is Hermitian, so are $X^HAX$ and $Y^HAY$. From the spectral decomposition, we have $X^HAX=U_1D_1U_{1}^{H}$, 
where $U_1$ is unitary and $D_1$ is diagonal. Similarly, $Y^HAY=U_2D_2U_{2}^{H}$, where $U_2$ is unitary and $D_2$ is diagonal.
As a consequence, we have
\begin{eqnarray*} 
\Lambda\lp  X^HAX\rp -\Lambda\lp  Y^HAY\rp &=&\Lambda\lp  X^HAX\rp -\Lambda\lp  \lp  X^HY\rp Y^HAY\lp  X^HY\rp ^{-1}\rp \\
&=&\Lambda\lp  U_1D_1U_1^H\rp -\Lambda\lp  \lp  X^HY\rp U_2D_2U_2^H\lp  X^HY\rp ^{-1}\rp .
\end{eqnarray*}
Applying Theorem \ref{thm:condition}, we obtain 
\begin{eqnarray*}
&&\begin{array}{l}|\Lambda\lp  X^HAX\rp -\Lambda\lp  Y^HAY\rp | \end{array}\\ &&\prec_w\left[\kappa\lp U_1\rp \kappa\lp X^HYU_2\rp \right]^{1/2}S\lp  X^HAX-\lp  X^HY\rp Y^HAY\lp  X^HY\rp ^{-1}\rp .
\end{eqnarray*}  
 Furthermore, the condition number of $U_1$ is $1$ and the condition number of $X^HYU_2$ 
 is equal to the condition number of $X^HY$, i.e.
$\kappa\lp U_1\rp =1$ and $\kappa\lp X^HYU_2\rp =\kappa\lp X^HY\rp .$
Moreover, we have
\begin{eqnarray*}
&&X^HAX-\lp  X^HY\rp Y^HAY\lp  X^HY\rp ^{-1}\\
&&=\left[X^HAX\lp  X^HY\rp -\lp  X^HY\rp Y^HAY\right]\lp  X^HY\rp ^{-1}.
\end{eqnarray*}
Consequently, 
\begin{eqnarray}\label{eqn:ch6_3}
&&\begin{array}{l}\left|\Lambda\lp  X^HAX\rp -\Lambda\lp  Y^HAY\rp \right| \end{array}\nonumber\\ 
&&\prec_w\left[\kappa\lp X^HY\rp \right]^{1/2}S\left[X^HAX\lp  X^HY\rp -\lp  X^HY\rp Y^HAY\right]\,  S\lp  \lp  X^HY\rp ^{-1}\rp .
\end{eqnarray}
Substituting \eqref{eqn:maj8} and \eqref{eqn:tansq1} into \eqref{eqn:ch6_3} completes the proof.
\end{proof}

The bounds in Theorems \ref{thm:mixsincos}, \ref{thm:squaretan_1}, and 
\ref{thm:conditiontan2} are different, but comparable and, in some particular cases, actually the same.
 For example, the bounds for the largest component
$\max\left|\Lambda\lp X^HAX\rp -\Lambda\lp Y^HAY\rp \right|$
are the same in both Theorem~\ref{thm:mixsincos} and Theorem~\ref{thm:squaretan_1}. 
If $\dim\lp \X\rp =\dim\lp \Y\rp =1$, then the bounds in all three theorems are the same as the 
first inequality in \eqref{eqn:chap6_1}.

Theorem \ref{thm:conditiontan2} may be tighter for the sum of $\left|\Lambda\lp X^HAX\rp -\Lambda\lp Y^HAY\rp \right|$, compared to that of Theorems \ref{thm:mixsincos} and \ref{thm:squaretan_1}, since 
\[
1/\cos\lp \Theta^\downarrow\lp \X,\Y\rp \rp \leq [1/\cos\lp \theta_{\max}\lp \X,\Y\rp \rp ,\ldots,1/\cos\lp \theta_{\max}\lp \X,\Y\rp \rp ].
\]
 
Theorems \ref{thm:mixsincos}, \ref{thm:squaretan_1}, and~\ref{thm:conditiontan2} provide various alternatives to conjecture \eqref{eqn:conjecture}, all involving  the singular values of $P_\X R_Y$ and $P_\Y R_X$. Our second conjecture
\eqref{eqn:conjecture_tan} relies instead on the singular values of $P_{\X+\Y} R_Y$ and $P_{\X+\Y} R_X$. 
We next clarify the relationship between these singular values.

\begin{lemma}\label{lem:ch6_maj3}
We have
$
S\lp P_\X R_Y\rp \prec_w S\lp P_{\X+\Y}R_Y\rp \, \sin\lp \Theta\lp \X,\Y\rp \rp $ and, similarly,  $S\lp P_\Y R_X\rp \prec_w S\lp P_{\X+\Y}R_X\rp \, \sin\lp \Theta\lp \X,\Y\rp \rp .$
\end{lemma}
\begin{proof}
Since the singular values are invariant under unitary transforms and the matrix conjugate transpose, we get
$
S\lp P_\X R_Y\rp =S\lp X^H P_\YP AY\rp =S\lp Y^HAP_\YP X\rp .
$
The identities
$Y^HAP_\YP X=Y^HAP_\YP P_\YP X= Y^HAP_\YP P_{\X+\Y} P_\YP X$
hold, since 
\[
P_\YP X= X-P_\Y X= P_{\X+\Y}\lp  X- P_\Y X\rp = P_{\X+\Y} P_\YP X,
\]
where every column of the matrix $X-P_\Y X$ evidently belongs to the subspace $\X+\Y.$
Thus, Theorem \ref{thm:general} gives
\begin{eqnarray*}
S\lp Y^HAP_\YP X\rp &\prec_w& S\lp  Y^HAP_\YP P_{\X+\Y} \rp \,  S\lp P_\YP X\rp =
S\lp  P_{\X+\Y} P_\YP AY\rp \,  S\lp P_\YP X\rp \\
&=&S\lp P_{\X+\Y} R_Y\rp \,  S\lp P_\YP X\rp .
\end{eqnarray*}
The singular values of $P_\YP X$ coincide with the sines of the
principal angles between $\X$ and $\Y$, e.g.,\ see \cite{bjorck73,ka02}, since
$\dim\lp \X\rp =\dim\lp \Y\rp $, i.e. $S\lp P_\YP X\rp =\sin\lp \Theta\lp \X,\Y\rp \rp $. This proves
$S\lp P_\X R_Y\rp \prec_w S\lp P_{\X+\Y} R_Y\rp \sin\lp \Theta\lp \X,\Y\rp \rp $. 

The second bound similarly follows from 
$S\lp P_\Y R_X\rp \prec_w S\lp P_{\X+\Y} R_X\rp S\lp P_\XP Y\rp $
since $S\lp P_\XP Y\rp =\sin\lp \Theta\lp \X,\Y\rp \rp $ due to the symmetry of PABS 
and $\dim\X=\dim\Y$.
\end{proof}

We note that \eqref{eqn:conjecture} implies \eqref{eqn:conjecture_tan} using Lemma \ref{lem:ch6_maj3}. 
We can also combine Theorems~\ref{thm:mixsincos}, \ref{thm:squaretan_1}, and~\ref{thm:conditiontan2} with  Lemma \ref{lem:ch6_maj3}
to easily obtain several tangent-based bounds. 

\begin{cor}\label{thm:mixsincos1}
 Under the assumptions of Conjecture \ref{thm:conjecture},
we have 
\[
\left|\Lambda\lp X^HAX\rp -\Lambda\lp Y^HAY\rp \right| \prec_w
\left\{ S\lp P_{\X+\Y}R_X\rp +S\lp P_{\X+\Y}R_Y\rp \right\} \frac{\sin\lp \Theta\lp \X,\Y\rp \rp }{\cos\lp \theta_{\max}\lp \X,\Y\rp \rp },
\]
\[
\left|\Lambda\lp X^HAX\rp -\Lambda\lp Y^HAY\rp \right|^2  \prec_w
\left\{ S\lp P_{\X+\Y}R_X\rp +S\lp P_{\X+\Y}R_Y\rp \right\}^2 \,  \tan^2\lp \Theta\lp \X,\Y\rp \rp ,
\]
\[
\left|\Lambda\lp X^HAX\rp -\Lambda\lp Y^HAY\rp \right|  \prec_w
\sqrt{c}\,\left\{ S\lp P_{\X+\Y}R_X\rp +S\lp P_{\X+\Y}R_Y\rp \right\}\, \tan\lp \Theta\lp \X,\Y\rp \rp .
\]
In addition, if the subspace $\X$ is $A$-invariant, then  correspondingly we have
\begin{equation}\label{e:c1}
\left|\Lambda\lp X^HAX\rp -\Lambda\lp Y^HAY\rp \right|  \prec_w
S\lp P_{\X+\Y}R_Y\rp  \frac{\sin\lp \Theta\lp \X,\Y\rp \rp }{\cos\lp \theta_{\max}\lp \X,\Y\rp \rp },
\end{equation}
\begin{equation}\label{e:c2}
\left|\Lambda\lp X^HAX\rp -\Lambda\lp Y^HAY\rp \right|^2 \prec_w
S^2\lp P_{\X+\Y} R_Y\rp \, \tan^2\lp \Theta\lp \X,\Y\rp \rp ,
\end{equation}
\begin{equation}\label{e:c3}
\left|\Lambda\lp X^HAX\rp -\Lambda\lp Y^HAY\rp \right|\prec_{w}\sqrt{c}\,S\lp P_{\X+\Y} R_Y\rp \, 
\tan\lp \Theta\lp \X,\Y\rp \rp ,
\end{equation}
where $c=\cos\lp \theta_{\min}\lp \X,\Y\rp \rp /\cos\lp \theta_{\max}\lp \X,\Y\rp \rp $.
\end{cor}

\section{Discussion}\label{sec:discussion}
In this section, we briefly discuss and compare our new mixed majorization bounds of the absolute changes of  eigenvalues, e.g.,\ where
the subspace $\X$ is $A$-invariant, with some known results, formulated here in our notation.

 Our bounds \eqref{e:c1} and \eqref{e:c2} are stronger than the following particular cases, where $\dim\X=\dim\Y$, of 
 \cite[Theorem 1]{ovtchinnikov62},  
 \[
 \cos\lp \theta_{\max}\lp \X,\Y\rp \rp \max\left|\Lambda\lp X^HAX\rp -\Lambda\lp Y^HAY\rp \right|\leq \sin\lp \theta_{\max}\lp \X, \Y\rp \rp \|R_Y\|,
 \]
and \cite[Remark 3]{ovtchinnikov62}, that uses the Frobenius norm $\|\cdot\|_F$,
 \[
 \cos\lp \theta_{\max}\lp \X,\Y\rp \rp \left\|\Lambda\lp X^HAX\rp -\Lambda\lp Y^HAY\rp \right\|_F\leq \sin\lp \theta_{\max}\lp \X, \Y\rp \rp \|R_Y\|_F.
 \]
 
\subsection{Sun's 1991 majorization bound}
Substituting the $2$-norm in \cite[Theorem 3.3]{sun91} in our notation as follows,
 $\left\|I-X^HYY^HX\right\|_2=\sin^2\lp  \theta_{\max}\lp \X,\Y\rp \rp $, one obtains the bound
$\left|\Lambda\lp  X^HAX\rp -\Lambda\lp  Y^HAY\rp \right|\prec_w S\lp  R_Y\rp \tan\lp  \theta_{\max}\lp \X,\Y\rp \rp ,$
which is weaker compared to our bound \eqref{e:c1}, since 
\begin{eqnarray}\label{eqn:RY}
S\lp  P_{\X+\Y} R_Y\rp \leq s_\max\lp  P_{\X+\Y}\rp S\lp  R_Y\rp =S\lp  R_Y\rp 
 \end{eqnarray}
and
$\sin\lp  \Theta\lp \X,\Y\rp \rp \leq\sin\lp  \theta_{\max}\lp \X,\Y\rp \rp $. 

\subsection{First order a posteriori majorization bounds}
A posteriori majorization bounds in terms of norms of residuals for the Ritz values approximating eigenvalues are known for decades.  
We quote here one of the best such bounds of the first order, i.e. involving the norm rather, then the norm squared,
of the residual.
\begin{thm}[{\cite{{bhatia_book},{sunste},{xie97}}}]\label{thm:majorresidual}
Let $Y$ be an orthonormal $n$ by $p$ matrix and matrix $A$ be Hermitian. 
Then there exist a set of incices $1\leq i_1<i_2<\cdots<i_p\leq n$, 
and some $p$ eigenvalues of $A$ as $\Lambda_I\lp A\rp =\lp  \lambda_{i_1}, \ldots, \lambda_{i_p}\rp $, such that
\begin{eqnarray*}
 |\Lambda_I\lp A\rp -\Lambda\lp Y^HAY\rp |\prec_w [s_1,s_1,s_2,s_2,\ldots] \prec_w 2S\lp  R_Y\rp , 
 \end{eqnarray*}
 where $S\lp R_Y\rp =[s_1,s_2,\ldots,s_p]$. The multiplier $2$ cannot be removed; see \cite[p.~188]{bhatia_book}.
 \end{thm}

It is important to realize that in our bounds the choice of subspaces $\X$ and $\Y$ is arbitrary, while
in Theorem \ref{thm:majorresidual} one cannot choose the subspace $\X.$ 
The implication of this fact is that
we can choose $\X$ in our bounds, such that $\theta_{\max}\lp \X,\Y\rp \leq \pi/4$, to make our bounds sharper. Next, we describe some situations
where principal angles are less than $\pi/4.$

\begin{enumerate}
\item \cite{george00} Let $A$ be a Hermitian quasi-definite matrix, i.e.
\begin{eqnarray} 
	A&=&\left[\begin{array}{lr}
         H & B^H\\
         B & -G
\end{array}\right],
\end{eqnarray} 
where $H \in \C^{k\times k}$and $G \in \C^{\lp n-k\rp  \times \lp n-k\rp }$ are 
Hermitian positive definite matrices. Let the subspace $\X$ be spanned by the eigenvectors 
of $A$ corresponding to $p$ eigenvalues which have the same sign. 
Let the subspace $\Y$ be spanned by $e_1,\ldots, e_p$ and $\Z$ be spanned by 
$e_{n-p+1},\ldots, e_{n}$, where $e_i$ is the coordinate vector. Then,
if the eigenvalues corresponding to the eigenspace $\X$ are positive, we have 
$
\theta_{\max}\lp \X,\Y\rp  < \pi/4.
$
If the eigenvalues corresponding to the eigenspace $\X$ are negative, we have  
$
\theta_{\max}\lp \X,\Z\rp  < \pi/4.
$
	\item \cite[p.~64]{bosner09} Let $[\QY\, \QYP]$	be unitary. Suppose 
	$\lambda_{\max}\lp Y^HAY\rp  < \lambda_{\min}\lp \QYP^HA\QYP\rp $ and $\X$ is the space spanned 
	by the eigenvectors corresponding to the $p$ smallest eigenvalues of $A$. Then 
$	\theta_{\max}\lp \X,\Y\rp  < \pi/4$.
	\item \cite[$\sin\lp 2\theta\rp $ and Theorem 8.2]{kahan} 
	Let $A$ be Hermitian and let $[\QX\, \QXP]$	be 
unitary with $X\in \C^{n\times p}$, such that
$
[\QX\, \QXP]^HA[\QX\, \QXP]=\diag\lp L_1,L_2\rp .
$ 
 Let $Y \in \C^{n \times p}$ be with orthonormal columns and $H_Y=Y^HAY.$
Let there be $\delta >0$, such that $\Lambda\lp L_1\rp  \in [\alpha, \beta]$ and
	 $\Lambda\lp L_2\rp  \in \R\setminus[\alpha-\delta, \beta+\delta]$. 
	Let  $\Lambda\lp H_Y\rp  \in [\alpha-\frac{\delta}{2}, \beta-\frac{\delta}{2}]$. 
	Then $\theta_{\max}\lp \X,\Y\rp  < \pi/4$. 
\end{enumerate}

Theorem \ref{thm:majorresidual} gives a first order error bound using $S\lp  R_Y\rp $. 
Under additional assumptions, there are similar, but second order, also called quadratic, i.e. involving the square $S^2\lp  R_Y\rp $, a posteriori error bounds; e.g., \cite{Mathias98,ovtchinnikov62}.
Next we check how some known bounds (see, e.g.,\ \cite{kahan,Mol16,Nakatsukasa12,Xie1997} and references there) for the angles $\Theta\lp \X,\Y\rp$ in terms of $S\lp  R_Y\rp $
can be combined with our tangent-based results, leading to various 
second order a posteriori error bounds, comparable to those in \cite{Mathias98,ovtchinnikov62}. 

\subsection{Quadratic a posteriori majorization bounds}
The second order a posteriori error bounds involve the term $S^2\lp  R_Y\rp $ and a gap, 
see, e.g., \cite{Mathias98,ovtchinnikov62} and references there.
In \cite[$\sin\lp \theta\rp $ and $\tan\lp \theta\rp $ Theorem]{kahan} and its variations, e.g.,\ \cite{Mol16,Nakatsukasa12}, bounds  of $\sin\lp  \Theta\lp \X,\Y\rp \rp $ and $\tan\lp  \Theta\lp \X,\Y\rp \rp $ 
above in terms of $S\lp  R_Y\rp $ and the gap are derived.  Two known theorems bounding principal angles between an exact invariant subspace $\X$ and its approximation $\Y$ are presented below, the first one for the $\sin\lp  \Theta\lp \X,\Y\rp \rp $.
\begin{thm}[{\cite[$\sin{\Theta}$ Theorem]{kahan}}]\label{thm:residualsin}
Let $A$ be Hermitian and let $[\QX\, \QXP]$	be unitary with $\QX\in \C^{n\times p}$, such that
$[\QX\, \QXP]^HA[\QX\, \QXP]=\diag\lp L_1,L_2\rp.$
 Let $Y \in \C^{n \times p}$ have orthonormal columns and $R_Y=AY-YH_Y$ with $H_Y=Y^HAY.$ 
 If $\Lambda\lp H_Y\rp\subset [a, b]$ and $\Lambda\lp L_2 \rp\subset \R\setminus [a-\delta,b+\delta]$ with $\delta >0$, then
 \begin{eqnarray*}\label{eqn:sinthm1}
 \sin\lp  \Theta\lp \X,\Y\rp \rp \prec_w \frac{S \lp R_Y \rp}{\delta}.
 \end{eqnarray*}
\end{thm}

The $ \tan\lp\theta\rp$ theorem in \cite{kahan} is valid only for the
Ritz values of $A$ with respect to $Y$ above or below the eigenvalues of $L_2$.
However, in \cite[Theorem 1]{Nakatsukasa12} the conditions of
the $ \tan\lp\theta\rp$ theorem are relaxed, as quoted below.
\begin{thm}\label{thm:residual}
In the notation of Theorem \ref{thm:residualsin}, let
\begin{enumerate}
	\item\cite{kahan} $\Lambda\lp H_Y \rp\subset [a, b]$, while $\Lambda\lp L_2\rp\subset (-\infty,a-\delta]$ or $\Lambda\lp L_2\rp\subset [b+\delta, \infty)$, or 
	\item\cite{Nakatsukasa12} $\Lambda\lp L_2\rp\subset [a, b]$, 
while $\Lambda\lp H_Y\rp$ lies in the union of $(-\infty,a-\delta]$ and $ [b+\delta, \infty)$.
\end{enumerate} 
  Then, we have
 \[
 \tan\lp  \Theta\lp \X,\Y\rp \rp \prec_w \frac{S\lp R_Y\rp}{\delta}.
 \]
\end{thm}

Results of \cite[Section 6]{zhu2013} suggest that Theorem \ref{thm:residual}, but not \ref{thm:residualsin}, 
might actually be simply improved by substituting $S\lp P_{\X+\Y}R_Y\rp$ for $S\lp R_Y\rp$, but such an investigation is beyond the goals of this work. 

Theorems \ref{thm:residualsin} and \ref{thm:residual} combined with mixed majorization results in Corollary~\ref{thm:mixsincos1} lead to various
second order a posteriori error bounds, e.g.,
\begin{cor}\label{cor:squareresidul1}
We have
 \begin{eqnarray}\label{eqn:secondordersin}
\hspace{1cm}
 \left|\Lambda\lp X^HAX\rp -\Lambda\lp Y^HAY\rp \right|  \prec_w  
\frac{S\lp P_{\X+\Y}R_Y\rp S\lp R_Y\rp }{\cos\lp \theta_{\max}\lp \X,\Y\rp \rp\delta } \leq 
\frac{S^2\lp R_Y \rp}{\cos\lp \theta_{\max}\lp \X,\Y\rp \rp\delta },
 \end{eqnarray}
under the assumptions of Theorem \ref{thm:residualsin}, and, under the assumptions of Theorem \ref{thm:residual},
 \begin{eqnarray}\label{eqn:secondtanbound}
\left|\Lambda\lp X^HAX\rp -\Lambda\lp Y^HAY\rp\right|^2  \prec_w 
\frac{ S^2\lp P_{\X+\Y} R_Y\rp S^2\lp R_Y\rp }{\delta^2}\leq
\frac{ S^4\lp R_Y\rp}{\delta^2}.
 \end{eqnarray}
\end{cor}

Bound \eqref{eqn:secondordersin} implies \cite[Bound (21) in Theorem 2]{ovtchinnikov62}.

If we take $\Y$ being spanned by the first coordinate vectors $e_1,\ldots, e_p$ and if the eigenvalues of $X^HAX$ correspond to the extreme eigenvalues of $A$, then bound \eqref{eqn:secondtanbound}  
is stronger than \cite[Inequality (13) in Theorem 4]{Mathias98}.
Moreover, the alternative assumption (2) in  Theorem~\ref{thm:residual} is not discussed in \cite{Mathias98}.
 
\subsection{Bounds for matrix eigenvalues after discarding off-diagonal blocks}
Bounding the change of matrix eigenvalues after discarding off-diagonal blocks in the matrix is a traditional, and still active, topic of research, even for the simplest 2-by-2 block Hermitian case, starting with the classical results of Fan and Thompson; see, e.g.,\ \cite{marshall79}. We note a trivial, but important, fact that discarding off-diagonal blocks can be viewed as an application of the Rayleigh-Ritz method. Indeed, let $A$ be a Hermitian 2-by-2 block matrix,
\begin{eqnarray*} 
	A&=&\left[\begin{array}{lr}
         A_{11}  & A_{12}\\
         A_{21} & A_{22}
\end{array}\right],
\end{eqnarray*} 
where $A_{11} \in \C^{k\times k}$and $A_{22} \in \C^{\lp n-k\rp  \times \lp n-k\rp }$. Let the subspace $\X$ be spanned by any $k$ eigenvectors of $A$ and the subspace $\Y$ be spanned by $e_1,\ldots, e_k$ and $\YP$ be spanned by 
$e_{n-k+1},\ldots, e_{n}$, where $e_i$ is the coordinate vector, i.e.
\[ 
	X=\left[\begin{array}{c}
         X_{1}\\
         X_{2} 
\end{array}\right]
\text{ and }
Y=\left[\begin{array}{c}
         I\\
         0 
\end{array}\right], 
\text{ therefore, } 
\tan\lp \Theta\lp \X,\Y\rp \rp =S\lp  X_2 X_1^{-1}\rp 
\]
by \cite[pp.~231--232]{sunste}, assuming for simplicity that the matrix $X_1$ is invertible; for a general case, see \cite[Remark 3.1]{ZhuK13}.
Clearly, $\rho\lp Y\rp =Y^HAY=A_{11}$
and our bounds for $\left|\Lambda\lp  X^HAX\rp -\Lambda\lp  Y^HAY\rp \right|$ turn into bounds for the $k$ eigenvalues $\Lambda\lp  X^HAX\rp $ of $A$ after discarding off-diagonal blocks and looking only at the upper left block $Y^HAY=A_{11}$. 
For example, bound \eqref{eqn:eigen_conjecture_tan} has the right-hand side $S\lp P_{\X+\Y} R_Y\rp \, \tan\lp \Theta\lp \X,\Y\rp \rp $. 
We now use \eqref{eqn:RY} and also take into account that  
\[ 
	R_Y= AY-Y\rho\lp Y\rp =
	\left[\begin{array}{c}
         0\\
         A_{12} 
\end{array}\right],
\]
thus, 
$S\lp  R_Y\rp =S\lp  AY-Y\rho\lp Y\rp \rp =S\lp  A_{12}\rp $.  We conclude that bound \eqref{eqn:eigen_conjecture_tan} in this case implies  
\[\left|\Lambda\lp  X^HAX\rp -\Lambda\lp  Y^HAY\rp \right|\prec_w S\lp  A_{12}\rp S\lp  X_2 X_1^{-1}\rp .\]

Similarly, we can apply our bounds to the lower right block $A_{22}$ by appropriately redefining $Y$, leaving this application as an exercise for the reader. 

Our results appear to be novel in this context as well, relying on both $S\lp  A_{12}\rp $ and $S\lp  X_2 X_1^{-1}\rp $ and allowing one to choose the $k$-dimensional $A$-invariant subspace~$\X$. 
In contrast, traditional bounds compare all, rather than selected, eigenvalues, using $S\lp  A_{12}\rp $ that may be squared, if also the spectral gap between the spectra of $A_{11}$ and $A_{22}$ is involved; see, e.g., \cite{Mathias98} and references there. Some authors bound only the maximal change, without majorization, e.g., \cite{lili05}, using $s^2_\max\lp  A_{12}\rp $. 

We concur with an anonymous referee that an interesting task for future research could be investigating possible connections between and combinations of the traditional and our results, in addition to already 
performed in the previous subsection comparison with \cite[Inequality (13) in Theorem 4]{Mathias98}.

\subsection{Bounds for eigenvalues after matrix additive perturbations}\label{ss:ap}
Discarding off-diagonal blocks in the matrix is one specific instance of a matrix additive perturbation. 
Generally, one is interested in bounding $|\Lambda\lp F\rp -\Lambda\lp G\rp |$ in terms of $S\lp F-G\rp $, where $F$ and $G$ are Hermitian matrices, e.g., as in the classical Weyl's Theorem \ref{thm:bhatia1}. 
As we now demonstrate, bounds for the Rayleigh-Ritz method
are surprisingly still applicable, with only one extra assumption that both $F$ and $G$ are shifted and pre-scaled to have their spectra in $[0,1]$. We adopt below a technique from~\cite{ka06}, where a connection is discovered between the Ritz values and PABS. The~technique uses a dilation of an operator on a Hilbert space, defined as an operator on a larger Hilbert space, whose composition with the orthogonal projection onto the original space gives the original operator, after being restricted to the original space.

It is shown in \cite{h50,rn90}, and used in \cite{ka06}, that any $n\times n$ Hermitian matrix $F$ with the spectrum in $[0,1]$ can be extended in a space of the double dimension $2n$ to an orthogonal projector, specifically,
\begin{eqnarray} \label{eqn:CS}
	P\lp F\rp &=&\left[\begin{array}{lr}
         F  & \sqrt{F\lp I-F\rp }\\
         \sqrt{\lp I-F\rp F} & I-F
\end{array}\right]
=
\left[\begin{array}{lr}
         \sqrt{F} \\
         \sqrt{I-F} 
\end{array}\right]
\left[\begin{array}{lr}
         \sqrt{F}  & \sqrt{I-F}
\end{array}\right].
\end{eqnarray} 
It is actually the CS decomposition, see, e.g.,\ \cite{bhatia_book,sunste},
but we do not use this fact. Instead, 
introducing the orthogonal projector
\begin{eqnarray*} 
	P_\Z&=&\left[\begin{array}{lr}
         I  & 0\\
         0 & 0
\end{array}\right],
\end{eqnarray*} 
we observe from \eqref{eqn:CS} that the orthogonal projector $P\lp F\rp $ is the dilation of $F$, since evidently 
$F=\left.\lp  P_\Z P\lp F\rp \rp \right|_\Z$, where the subspace $\Z$ is spanned by the first $n$ coordinate vectors in the space of the double dimension. 

Let the subspace $\F$ denote the range of $P\lp F\rp $ in the space of the double dimension.
By an analog of Definition \ref{thm:2.1def} for restricted products of orthogonal projectors, see, e.g.,\ \cite[Lemma 2.8]{ka06}, we have
$\cos^2\lp  \Theta^{\uparrow}\lp \F,\Z\rp \rp =\Lambda\lp F\rp $.  Alternatively, setting $A=P_\Z$, 
we notice that $F$ can now be interpreted as a result of application of the Rayleigh-Ritz method 
for the Hermitian matrix $A=P_\Z$ on the test/trial subspace $\F$, where the Ritz values are thus given by 
$\Lambda\lp F\rp $!

Let $G$ be another $n\times n$ Hermitian matrix with the spectrum in $[0,1]$. Defining similarly 
the orthogonal projector $P\lp G\rp $ as the dilation of $G$ in the $2$-by-$2$ block form and its range $\G$,
we come to the required conclusion that $\Lambda\lp F\rp -\Lambda\lp G\rp $ is nothing but a change in the Ritz values 
where the test/trial subspace $\F$ turns into $\G$. Substituting $\F$ and $\G$ for $\X$ and $\Y$, 
our bounds for Ritz values are immediately applicable,
and give apparently new bounds for eigenvalues after matrix additive perturbations. 

For example, let us consider \eqref{eqn:conjecture_tan} with $\X=\F$ and $\Y=\G$, where 
for simplicity we drop, using the same arguments as in \eqref{eqn:RY}, the extra projector $P_{\F+\G}$, thus obtaining 
\begin{eqnarray}\label{eqn:conjecture_tan_simp}
\left|\Lambda\lp F\rp -\Lambda\lp G\rp \right|&\prec_{w}& 
 \left\{ S\lp R_\F\rp +S\lp R_\G\rp \right\}\, \tan\lp \Theta\lp \F,\G\rp \rp,
\end{eqnarray}
assuming $\dim\, \F=\dim\, \G$ for \eqref{eqn:conjecture_tan} to hold.
The matrix $A=P_\Z$ is the orthogonal projector. By Lemma \ref{thm:sin2theta}, up to zero values,  
\begin{eqnarray} \nonumber
S\lp R_\F\rp &=&S\lp  \lp I-P\lp F\rp \rp AP\lp F\rp \rp \\ \nonumber
&=&\lp  \sin\lp  \Theta\lp \F,\Z\rp \rp \cos\lp  \Theta\lp \F,\Z\rp \rp \rp ^\downarrow\\
&=& \lp  \sqrt{1-\Lambda\lp F\rp }\sqrt{\Lambda\lp F\rp }\rp ^\downarrow, \label{eqn:sl} 
\end{eqnarray}
and, similarly, $S\lp R_\G\rp =\lp  \sqrt{1-\Lambda\lp G\rp }\sqrt{\Lambda\lp G\rp }\rp ^\downarrow$.
Clearly, $S\lp R_\F\rp $ thus vanishes if $F$ is an orthogonal projector, i.e. with the spectrum 
$\Lambda\lp F\rp $ consisting of zeros and ones. 

Since the subspaces $\F$ and $\G$ have their orthonormal  bases
$[ \sqrt{F} \; \sqrt{I-F} ]^H$ and $[ \sqrt{G} \; \sqrt{I-G} ]^H$, correspondingly,
we obtain by Definition \ref{thm:2.1def} that
\begin{equation}\label{eqn:cos}
 \cos \Theta^\uparrow\lp \F,\G\rp  = S\lp   \sqrt{F}\sqrt{G} + \sqrt{I-F}\sqrt{I-G} \rp .
\end{equation}
We notice that, if the matrices $F$ and $G$ are close to each other, then 
the $\cos \Theta^\uparrow\lp \F,\G\rp $ vector \eqref{eqn:cos} is made of nearly ones, 
so that the $\tan \Theta\lp \F,\G\rp $ multiplier in \eqref{eqn:conjecture_tan_simp} almost vanishes,
i.e. plays the same role as the term $S\lp F-G\rp $ in Weyl's Theorem \ref{thm:bhatia1}.

Thus, we finally have all the values in the right-hand side of bound \eqref{eqn:conjecture_tan_simp} 
expressed only in terms of the original matrices $F$ and $G$, leading to the new bound
of sensitivity of eigenvalues of Hermitian matrices with respect to additive perturbation.
Since the bound is a bit bulky, we let the reader to put all the pieces of the puzzle together.

The most benefiting for our bound case is
where both matrices $F$ and $G$ are approximate orthogonal projectors, so that 
both multipliers in the right-hand side of bound~\eqref{eqn:conjecture_tan_simp} can be small.
Indeed, let $\left\|F-\bar{F}\right\|=\epsilon_F$ and $\left\|G-\bar{G}\right\|=\epsilon_G$, 
where $\bar{F}$ and $\bar{G}$ are orthogonal projectors. 
Then the spectrum $\Lambda(F)$ ($\Lambda(G)$) consists of only zeros and ones, with
$\epsilon_F$ ($\epsilon_G$ correspondingly) accuracy. 
Thus, 
$S\lp R_\F\rp =O\lp\sqrt{\epsilon_F}\rp$ and $S\lp R_\G\rp =O\lp\sqrt{\epsilon_G}\rp$ 
by \eqref{eqn:sl}.
We can also simplify the right-hand side of \eqref{eqn:cos},  
since 
\begin{eqnarray*}
\sqrt{F}\sqrt{G} + \sqrt{I-F}\sqrt{I-G} &\approx& \sqrt{\bar{F}}\sqrt{\bar{G}} + \sqrt{I-\bar{F}}\sqrt{I-\bar{G}}\\
&=& \bar{F}\bar{G} +(I-\bar{F})(I-\bar{G}) \\
&=& (I-\bar{F}) -\bar{G} 
\end{eqnarray*}
with $O\lp\epsilon_F\rp + O\lp\epsilon_G\rp$ accuracy.
By \cite[Theorem 2.17]{kja10} on eigenvalues of a difference of
orthogonal projectors, $\Lambda\lp (I-\bar{F}) -\bar{G} \rp = \pm\sin \Theta \lp \bar{\F}^\perp, \bar{\G} \rp$, 
excluding the values $\pm1$ and zeros. Also by \cite[Theorem 2.17]{kja10}, the multiplicity of zero in 
the spectrum $\Lambda\lp (I-\bar{F}) -\bar{G} \rp$ is 
$\dim\lp\bar{\F}^\perp\cap\bar{\G}\rp+\dim\lp\bar{\F}\cap\bar{\G}^\perp\rp$.
Finally, using \cite[Theorem 2.7]{kja10}, we have that 
$\sin \Theta \lp \bar{\F}^\perp, \bar{\G} \rp = \cos \Theta \lp \bar{\F}, \bar{\G} \rp$,
excluding zeros and ones. 

Collecting everything together, we obtain the following asymptotic bound,
\[
\left|\Lambda\lp F\rp -\Lambda\lp G\rp \right|\prec_{w} 
 \left\{ O\lp\sqrt{\epsilon_F}\rp + O\lp\sqrt{\epsilon_G}\rp \right\}\, \tan\lp \Theta\lp \bar{\F},\bar{\G}\rp \rp ,
\]
assuming that   $\dim\,\bar{\F}=\dim\,\bar{\G}$ and $\tan\lp \Theta\lp \bar{\F},\bar{\G}\rp\rp<\infty$.

In contrast, Weyl's Theorem \ref{thm:bhatia1} in this case gives 
\begin{eqnarray*}
 \left|\Lambda\lp F\rp -\Lambda\lp G\rp \right|&\prec_{w}& S(F-G)\\ 
 &\prec_{w}&
 \epsilon_F + \epsilon_G + S\lp\bar{F}-\bar{G}\rp,
\end{eqnarray*}
where  
$ S\lp\bar{F}-\bar{G}\rp = \left[ \sin\Theta\lp\bar{\F},\bar{\G}\rp ,\, \sin\Theta\lp\bar{\F},\bar{\G}\rp \right]^\downarrow$
by \cite[Theorem 2.17]{kja10}, up to zeros if $\dim\,\bar{\F}=\dim\,\bar{\G}$ and $\tan\lp \Theta\lp \bar{\F},\bar{\G}\rp\rp<\infty$.
We have $\Lambda\lp F\rp -\Lambda\lp G\rp \to 0$, which our bound exactly captures, having the right-hand side vanishing, as $\epsilon_F\to0$ and $\epsilon_G\to0$, while the right-hand side in 
Weyl's Theorem \ref{thm:bhatia1} does not vanish. 

\begin{figure}
\centering
\includegraphics[width=0.5\textwidth]{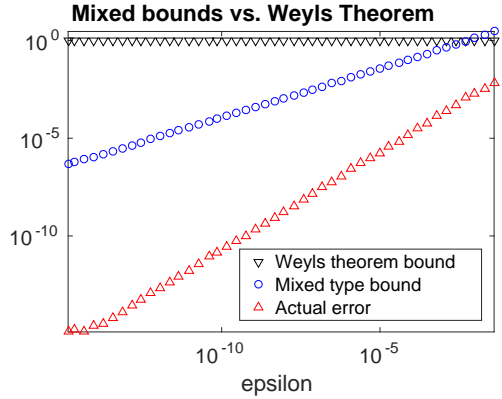}
\caption{Additive perturbed matrix eigenvalues bounded by \eqref{eqn:conjecture_tan_simp} vs.  Weyl's Theorem.}
\label{fig}
\end{figure}
To verify numerically our asymptotic arguments above, we test additive random perturbations of two $2$-by-$2$ orthogonal projectors.
The maximal changes in the eigenvalues and their bounds by the maximal bound  \eqref{eqn:conjecture_tan_simp} and Weyl's Theorem are 
plotted in Figure \ref{fig} as functions of $\epsilon=\epsilon_F=\epsilon_G$ in the log-log scale. We observe in Figure~\ref{fig} the predicted 
dependence of the right-hand side of  \eqref{eqn:conjecture_tan_simp} on the square root~$\sqrt{\epsilon}$. The~bound of Weyl's Theorem 
remains practically constant at the top of Figure \ref{fig}, 
also as predicted, and gets outperformed by \eqref{eqn:conjecture_tan_simp} if $\epsilon<10^{-1}$. 
Finally, both our \eqref{eqn:conjecture_tan_simp} and Weyl's Theorem bound above the actual changes in the eigenvalues, which surprisingly behave 
like ${\epsilon}$ showing that our $\sqrt{\epsilon}$ bound is not sharp in this situation.

\section{Conclusions}
We formulate a conjecture and derive several majorization-type mixed bounds for the
absolute changes of eigenvalues of the matrix RQ for Hermitian matrices in terms of PABS and singular values of residual matrices. Our~results improve and generalize known bounds.
We apply our mixed bounds involving the matrix RQ to classical problems of bounding 
changes in eigenvalues of Hermitian matrices under additive perturbations, and surprisingly obtain new and better results,
e.g.,\ outperforming in some cases the classical Weyl's Theorem. 

\section*{Acknowledgments} 
We are thankful to the members of the Ph.D. committee, Professors 
Julien Langou, Ilse Ipsen, Dianne O'Leary, and Beresford Parlett,
who have read and commented on the preliminary version of some of these 
results as appeared in 2012 in the Ph.D. thesis \cite{zhu2012thesis} of the first author, being advised by the second author. 
Two~anonymous referees have made numerous useful and motivating suggestions to improve the paper. We thank the referees and the Associate Editor Prof. Zlatko~Drmac for fast and careful handling of our manuscript. 

\appendix
\section*{Appendix}

The following theorem links unitarily invariant norm inequalities and weak majorization inequalities.
                           
\begin{thm}[{\cite[Corollary 7.4.47]{horn91} and \cite[p.~368]{marshall79}}]
Let $A$ and $B$ be square matrices. Then $S\lp  A\rp \prec_w S\lp  B\rp  $ if and only if $|||A|||\leq |||B|||$ 
for every unitarily invariant norm.
\end{thm}

Next, we provide some well known majorization results as they serve as  foundations for our
new majorization inequalities for singular values and eigenvalues in Lemma \ref{lemma:commutor2}.
\begin{thm}[{\cite[Fan Theorem p.~330]{marshall79}}]\label{thm:fan}
Let $A$ and $B$ be square matrices. Then 
\[S\lp A+B\rp \prec_w S\lp A\rp +S\lp B\rp .\]
\end{thm}
\begin{thm}[{\cite{ka10}}]\label{thm:general}
For general matrices $A$ and $B$, we have 
\[S\lp AB\rp  \prec_w S\lp A\rp \,  S\lp B\rp,\] 
$S\lp AB\rp  \leq s_\max\lp A\rp  S\lp B\rp $ and $S\lp AB\rp  \leq S\lp A\rp \, s_\max\lp B\rp.$
If needed, we  add zeros to match the sizes of vectors on either side.
\end{thm}
\begin{thm}[{\cite[p.~71, Theorem III.4.4]{bhatia_book} and \cite{ka06}}]\label{thm:bhatia1}
For Hermitian matrices $A$ and $B,$ we have
\[
|\Lambda\lp A\rp -\Lambda\lp B\rp | \prec_w S\lp A-B\rp .
\]
\end{thm}
\begin{thm}[{\cite[p.~236, Theorem VIII.3.9]{bhatia_book}}]\label{thm:condition}
Let $A$ and $B$ be square matrices such that $A=JD_1J^{-1}$ and $B=TD_2T^{-1},$ 
where $J$ and $T$ are invertible matrices, and $D_1$ and $D_2$ are real diagonal matrices. Then 
$
|\Lambda\lp A\rp -\Lambda\lp B\rp |\prec_w \left[\kappa\lp J\rp \kappa\lp T\rp \right]^{1/2}S\lp A-B\rp,
$
where $\diag\lp \Lambda\lp A\rp \rp $denotes a diagonal matrix whose diagonal entries arranged in decreasing order are the eigenvalues of A, and  $\kappa\lp \cdot\rp $ denotes the condition number.
\end{thm}

\begin{thm}[{\cite{bhatia97}}]\label{thm:half}
Let $A$ and $B$ be square matrices and $AB$ be a normal matrix. Then
$
S\lp  |AB|_{\pol}^{1/2}\rp  \prec_w S\lp  |BA|_{\pol}^{1/2}\rp,
$
where $|A|_{\pol}=\lp A^HA\rp ^{1/2}$. 
\end{thm}
\begin{thm}[{\cite[p.~94, Theorem IV.2.5]{bhatia_book}}]\label{thm:bhatia2}
Let $A$ and $B$ be square matrices. Then, for any $t>0$ we have
$S^t\lp AB\rp  \prec_w S^t\lp A\rp \,  S^t\lp B\rp,$
where $S^t\lp A\rp =[s^t_1\lp A\rp ,\ldots,s^t_n\lp A\rp ].$
\end{thm}
\begin{lemma}\label{lemma:half} 
Let $A$ and $B$ be square matrices. Then
\[S\lp  |AB|_{\pol}^{1/2}\rp  \prec_w S^{1/2}\lp A\rp \,  S^{1/2}\lp B\rp.\]
\end{lemma}
\begin{proof}
First, we show that $S\lp  |T|_{\pol}^{1/2}\rp =S^{1/2}\lp T\rp $ for all $T\in \C^{n\times n}.$
Since $|T|_{\pol}$ is semi-positive definite and Hermitian, it follows that
$\Lambda^{1/2}\lp  |T|_{\pol}\rp =S^{1/2}\lp  |T|_{\pol}\rp $.  
Also, it follows that $|T|_{\pol}^{1/2}$ is semi-positive definite and Hermitian.
Therefore, 
\[
S\lp  |T|_{\pol}^{1/2}\rp =\Lambda\lp  |T|_{\pol}^{1/2}\rp =\Lambda^{1/2}\lp  |T|_{\pol}\rp .
\]
Moreover, it is easy to check that $S\lp |T|_{\pol}\rp =S\lp T\rp .$
So, we have 
\[
S\lp  |T|_{\pol}^{1/2}\rp =\Lambda^{1/2}\lp  |T|_{\pol}\rp =S^{1/2}\lp T\rp .
\]
Setting $T=AB,$ we obtain 
$S\lp  |AB|_{\pol}^{1/2}\rp =S^{1/2}\lp AB\rp.$
Applying Theorem \ref{thm:bhatia2} for $S^{1/2}\lp AB\rp $ concludes the proof.
\end{proof}
\begin{thm}[{\cite[p.~254, Proposition IX.1.2]{bhatia_book}}]\label{thm:real}
Let $A$ and $B$ be square matrices and $AB$ be a Hermitian matrix. Then
$S\lp AB\rp  \prec_w S\lp  \re\lp BA\rp \rp,$
where $\re\lp A\rp $ denotes $\lp A+A^H\rp /2.$
\end{thm}
\begin{lemma}\label{lemma:realpart}
For a square matrix $A$, we have 
$S\lp  \re\lp A\rp \rp \prec_w S\lp A\rp.$
\end{lemma}
\begin{proof}
From Theorem \ref{thm:fan} we have
\[
S\lp \re\lp A\rp \rp =S\lp  \frac{A+A^H}{2}\rp \prec_w S\lp  \frac{A}{2}\rp +S\lp  \frac{A^H}{2}\rp .
\]
Since $S\lp A\rp =S\lp A^H\rp $, we have $S\lp  \re\lp A\rp \rp \prec_w S\lp A\rp .$
\end{proof}
\begin{thm}[{\cite{bhatia91}}]\label{thm:s2}
Let $A$ and $B$  be Hermitian  and $T$ be positive definite. Then 
\[s_{\min}\lp T\rp S\lp A-B\rp \prec_w S\lp AT-TB\rp.\]
\end{thm}

The following theorem is inspired by Bhatia, Kittaneh and Li in \cite{bhatia97}. 
We  prove the majorization inequalities 
bounding the singular values of $A-B$ by the singular values of $AT-TB$ and $T^{-1}A-BT^{-1}.$
\begin{thm}\label{thm:commutor}
Let $A$ and $B$  be Hermitian and $T$ be positive definite. Then
\[
S^2\lp A-B\rp  \prec_w S\lp AT-TB\rp \,  S\lp  T^{-1}A-BT^{-1}\rp .
\]
\end{thm}
\begin{proof}
Since  $A$ and $B$ are Hermitian, we have $\left|\lp  A-B\rp ^2\right|_{\pol}=\lp A-B\rp ^2$. Thus,
\[
S\lp  A-B\rp =S\lp  \left|\lp  A-B\rp ^2\right|_{\pol}^{1/2}\rp =S\lp  \left|\lp  A-B\rp T^{-1/2}T^{1/2}\lp  A-B\rp \right|_{\pol}^{1/2}\rp .
\]
Applying first Theorem \ref{thm:half} and then Lemma \ref{lemma:half}, we obtain
\begin{eqnarray*}
S\lp A-B\rp &\prec_w & S\lp  \left|T^{1/2}\lp  A-B\rp ^2T^{-1/2}\right|_{\pol}^{1/2}\rp \label{eqn:major1}\\
&\prec_w &S^{1/2}\lp  T^{1/2}\lp  A-B\rp T^{1/2}\rp \,  S^{1/2}\lp  T^{-1/2}\lp  A-B\rp T^{-1/2}\rp .
\end{eqnarray*} 
Since the increasing convex functions preserve the weak majorization, e.g.,\ $x^2$ for $x\geq0,$ 
we square both sides to get 
\begin{eqnarray*}
S^2\lp A-B\rp &\prec_w& S\lp  T^{1/2}\lp  A-B\rp T^{1/2}\rp \,  S\lp  T^{-1/2}\lp  A-B\rp T^{-1/2}\rp .
\end{eqnarray*} 
Using Theorem \ref{thm:real},
\begin{eqnarray}\label{eqn:comm1}
S^2\lp A-B\rp &\prec_w& S\lp  \re\left[\lp  A-B\rp T\right]\rp \,  S\lp  \re\left[T^{-1}\lp  A-B\rp \right]\rp .
\end{eqnarray} 
Matrices $BT-TB$ and $BT^{-1}-T^{-1}B$
 are skew-Hermitian, i.e. $C^H=-C$, thus,
 \begin{eqnarray}\label{eqn:comm3}
\re\lp  AT-TB\rp &=&\re\left[\lp  A-B\rp T+\lp  BT-TB\rp \right]
=\re\left[\lp  A-B\rp T\right]
\end{eqnarray} 
and
\begin{eqnarray}\label{eqn:comm2}
\re\lp  T^{-1}A-BT^{-1}\rp &=&\re\left[T^{-1}\lp  A-B\rp -\lp  BT^{-1}-T^{-1}B\rp \right]\nonumber\\
&=&\re\left[T^{-1}\lp  A-B\rp \right].
\end{eqnarray} 
By Lemma \ref{lemma:realpart}, 
\begin{eqnarray}\label{eqn:comm4}
S\left[\re\lp  AT-TB\rp \right]&\prec_w& S\lp  AT-TB\rp ,
\end{eqnarray} 
and
\begin{eqnarray}\label{eqn:comm5}
S\left[\re\lp  T^{-1}A-BT^{-1}\rp \right]&\prec_w& S\lp  T^{-1}A-BT^{-1}\rp .
\end{eqnarray} 
Plugging in \eqref{eqn:comm3}, \eqref{eqn:comm2},  \eqref{eqn:comm4}, and \eqref{eqn:comm5} into \eqref{eqn:comm1} completes the proof.
\end{proof}
\begin{remark}
In \cite{bhatia97}, the authors prove the inequality for unitarily invariant norm such that $|||A-B|||^2\leq |||AT-TB|||\,|||T^{-1}A-BT^{-1}||| $, which is equivalent to, see, e.g.,\ \cite[Theorem 2.1]{bhatia97},
\[
\lp  \sum_{i=1}^ks_i\lp A-B\rp \rp ^2 \leq \lp  \sum_{i=1}^ks_i\lp AT-TB\rp \rp \lp  \sum_{i=1}^ks_i\lp  T^{-1}A-BT^{-1}\rp\rp,  k=1,2\ldots,n. 
\]
In contrast, the weak majorization inequalities in our Theorem \ref{thm:commutor} mean 
\[
\sum_{i=1}^ks_i^2\lp A-B\rp  \leq \sum_{i=1}^k\lp  s_i\lp AT-TB\rp s_i\lp  T^{-1}A-BT^{-1}\rp \rp,  k=1,2\ldots,n.
\]
\end{remark}

\begin{lemma}\label{lemma:commutors}
Under the assumptions of Theorem \ref{thm:commutor}, we have
\[S^2\lp  A-B\rp \prec_w S^2\lp  T^{-1}\rp \,  S^2\lp  AT-TB\rp.\]
\end{lemma} 
\begin{proof}
Since $T^{-1}A-BT^{-1}=T^{-1}\lp  AT-TB\rp T^{-1}$, we have 
$$S\lp  T^{-1}A-BT^{-1}\rp \prec_w S\lp  T^{-1}\rp \,  S\lp  AT-TB\rp  \,  S\lp  T^{-1}\rp .$$
By Theorem \ref{thm:commutor}, this lemma is proved.
\end{proof}

\begin{lemma}\label{lemma:commutor2}
Let $A$ and $B$ be Hermitian and $T$ be invertible. We have
 \begin{enumerate}
\item $s_{\min}\lp T\rp \left|\Lambda\lp A\rp -\Lambda\lp B\rp \right|\prec_w S\lp  AT-TB\rp .$
\item	
$\left|\Lambda\lp A\rp -\Lambda\lp B\rp \right|^2\prec_w S\lp  AT-TB\rp \,  S\lp  T^{-1}A-BT^{-1}\rp .$
\item 
$\left|\Lambda\lp A\rp -\Lambda\lp B\rp \right|^2\prec_w S^2\lp  T^{-1}\rp \,  S^2\lp  AT-TB\rp .$
\end{enumerate}
\end{lemma}
\begin{proof}
 Let the SVD of $T$ be $U\Sigma V^H$,
 where $U$ and $V$ are $n$ by $n$ unitary matrices and $\Sigma$ is diagonal. 
Since $T$ is invertible, it follows that $\Sigma$ is positive definite. 
\[
AT-TB=AU\Sigma V^H-U\Sigma V^HB
    =U\lp  U^HAU\Sigma-\Sigma V^HBV\rp V^H.
\]
Since the singular values are invariant under unitary transforms, we have
\begin{eqnarray}\label{eqn:com3}
S\lp  AT-TB\rp =S\lp  U^HAU\Sigma-\Sigma V^HBV\rp .
\end{eqnarray}
Taking $A:=U^HAU$, $B:=V^HBV$, and $T:=\Sigma$ in Theorem \ref{thm:s2}, we obtain
\begin{eqnarray}\label{eqn:com1}
s_{\min}\lp \Sigma\rp S\lp U^HAU-V^HBV\rp \prec_w S\lp  U^HAU\Sigma-\Sigma V^HBV\rp .
\end{eqnarray}
Applying Theorem \ref{thm:bhatia1}, we have 
\begin{eqnarray}\label{eqn:com2}
|\Lambda\lp A\rp -\Lambda\lp B\rp |=|\Lambda\lp U^HAU\rp -\Lambda\lp V^HBV\rp |\prec_w S\lp  U^HAU-V^HBV\rp .
\end{eqnarray}
Combining \eqref{eqn:com3}, \eqref{eqn:com1}, and \eqref{eqn:com2}, we obtain the first statement. 

Next, we prove the second statement. We have
\[
T^{-1}A-BT^{-1}=V\Sigma^{-1} U^HA-BV\Sigma^{-1} U^H
    =V\lp  \Sigma^{-1} U^HAU-V^HBV\Sigma^{-1}\rp U^H.
\]
Therefore, we have $S\lp  T^{-1}A-BT^{-1}\rp =S\lp  \Sigma^{-1} U^HAU-V^HBV\Sigma^{-1}\rp .$

By Theorem \ref{thm:commutor}, we have
\[
S^2\lp  U^HAU-V^HBV\rp \prec_w S\lp  U^HAU\Sigma-\Sigma V^HBV\rp \,  S\lp  \Sigma^{-1} U^HAU-V^HBV\Sigma^{-1}\rp .
\]
By using \eqref{eqn:com2}, the second statement is proved. 
Similarly,  applying Lemma~\ref{lemma:commutors}
we obtain the third statement. 
\end{proof}

\begin{lemma}\label{thm:sin2theta}
Let $P$ and $Q$ be orthogonal projectors onto subspaces $\PP$ and $\QQ$. Then, 
up to zero values,  
$S\lp\lp I-P\rp QP\rp =\lp  \sin\lp  \Theta\lp \PP,\QQ\rp \rp \cos\lp  \Theta\lp \PP,\QQ\rp \rp \rp ^\downarrow.$
\end{lemma}
\begin{proof}
Let $\left[\QX, \QXP\right] $ be a unitary matrix and columns of the matrices $\QX $ and 
$\QY $   
form  orthonormal bases for the subspaces $\PP$ and $\QQ$, 
correspondingly. 
Then, up to zeros,
$
S\lp \lp I-P\rp QP\rp=S\lp\QXP\QXP^{H}\QY\QY^{H}\QX\QX^H\rp=S\lp\QXP^{H}\QY\QY^{H}\QX\rp.
$
By CS-decomposition, e.g., in notation of \cite[Formula (1)]{ZhuK13}, there exist unitary matrices 
$U_1, U_2$ and $V_1$, such that
$\QXP^{H}\QY=U_2^{H}\diag\lp 0,\sin\lp  \Theta\rp, 1 \rp V_1$  
and $\QY^{H}\QX=V_1^H\diag\lp 1,\cos\lp \Theta\rp , 0\rp U_1.$
Thus, 
$\QXP^{H}\QY\QY^{H}\QX=U_2^{H}\diag\lp 0,\sin\lp \Theta\rp\cos\lp \Theta\rp , 0\rp U_1,$
where $0$ and $1$ denote vectors of zeros and ones, correspondingly, of matching sizes, and 
$\Theta$ is the vector of the angles  in $\lp 0,\pi/2\rp$ between the subspaces $\PP$ and $\QQ$. 
Dropping again, if needed, the zero entries, gives the statement of the lemma.
\end{proof}

\vskip12pt
\def\refname{

\end{document}